\documentclass[11pt]{amsart}

\usepackage{lmodern}
\usepackage{amsmath,amssymb,amsthm,mathtools,amscd}
\usepackage{enumitem}
\usepackage{microtype}
\usepackage{needspace}
\usepackage[hidelinks]{hyperref}

\newcommand{\PP}{\mathbb P}
\newcommand{\QQ}{\mathbb Q}
\newcommand{\RR}{\mathbb R}
\newcommand{\ZZ}{\mathbb Z}
\newcommand{\CC}{\mathbb C}
\newcommand{\Sym}{\operatorname{Sym}}
\newcommand{\Pic}{\operatorname{Pic}}
\newcommand{\Hom}{\operatorname{Hom}}

\theoremstyle{plain}
\newtheorem{theorem}{Theorem}[section]
\newtheorem{proposition}[theorem]{Proposition}
\newtheorem{lemma}[theorem]{Lemma}
\newtheorem{corollary}[theorem]{Corollary}
\theoremstyle{remark}
\newtheorem{remark}[theorem]{Remark}

\title[Odd-degree rigidity and a characterization of complete flags]{Odd-Degree Rigidity and a Numerical Characterization of Complete Flag Varieties}
\author{Kiwamu Watanabe}
\date{\today}
\address{Department of Mathematics, Faculty of Science and Engineering, Chuo University.
1-13-27 Kasuga, Bunkyo-ku, Tokyo 112-8551, Japan}
\email{watanabe@math.chuo-u.ac.jp}
\subjclass[2020]{Primary 14M15; Secondary 14J45, 14E30, 57R20}
\keywords{Fano manifold, complete flag variety, maps of odd degree, homotopy rigidity, quadratic intersection relation, extremal contraction}

\begin{document}

\begin{abstract}
Let $F=G/B$ be the complete flag variety of a connected semisimple
complex linear algebraic group. We prove that if a smooth Fano variety $X$
has the same dimension and Picard number as $F$, then the existence of a
continuous map $X\to F$ of odd topological degree implies $X$ is
isomorphic to $F$.
In particular, a smooth Fano variety homotopy equivalent to $F$ is
isomorphic to $F$.

The algebraic input is a numerical characterization of complete flags:
a positive-dimensional smooth Fano variety $X$ is a complete flag variety
if and only if its pseudoindex $\iota_X$ is at least two and
$\ker(\Sym^2N^1(X)_\RR\to N^2(X)_\RR)$ contains a positive-definite
quadratic form. The pseudoindex condition can equivalently be replaced
by the divisibility of $-K_X$ by two in $\Pic(X)$. The key observation is
that a single definite quadratic relation excludes surfaces in every fiber of any 
elementary contraction; standard contraction theorems then yield smooth
$\PP^1$-fibrations and the characterization of complete flags applies.
For the odd-degree theorem, Poincar\'e duality and a mod-two Wu-class
argument show that the anticanonical class is divisible by two;
comparison of the integral coroot lattices then identifies the flag.
Most proofs were developed through extensive discussions with OpenAI's ChatGPT (GPT-5.6 Sol and GPT-6), based on mathematical directions supplied by the author and arguments suggested by ChatGPT; see the AI Disclosure.
\end{abstract}

\maketitle

\section{Introduction}

A \emph{complete flag variety} is a quotient $F=G/B$, where $G$
is a connected semisimple complex linear algebraic group and $B$ a
Borel subgroup. Let $X$ be a smooth Fano variety and let $f:X\to F$
be a map continuous in the classical (complex-analytic) topology;
$f$ need not be algebraic. Write $\rho(Y)$ for the Picard
number of a smooth projective variety $Y$. If
$\dim_{\CC}X=\dim_{\CC}F=n$, the complex structures orient the
underlying compact real $2n$-manifolds, and their fundamental classes
$[X]$, $[F]$ determine the \emph{topological degree} $\deg f\in\ZZ$
uniquely by $f_*[X]=(\deg f)[F]$ in $H_{2n}(F,\ZZ)$.
Our first result is the following rigidity theorem.
\begin{theorem}\label{thm:odd-degree-main}
Let $F=G/B$ be a complete flag variety and let $X$ be a smooth Fano
variety with $\dim X=\dim F$ and $\rho(X)=\rho(F)$. If there exists a
continuous map $f:X\to F$ of odd topological degree, then $X$ is
isomorphic to $F$. In particular, a smooth Fano variety homotopy equivalent to $F$ is
isomorphic to $F$.
\end{theorem}

Let $N^k(X)_\QQ$ denote codimension-$k$ cycles modulo numerical
equivalence. For each smooth projective variety $X$, put $r=\rho(X)$
and fix a basis $D_1,\ldots,D_r$ of $N^1(X)_\QQ$, also regarded as a
basis over $\RR$. We use this convention throughout, specifying any
alternative basis when needed. Define $R_X:=\ker\mu_X$, where
\[
\mu_X:\Sym^2N^1(X)_\QQ\longrightarrow N^2(X)_\QQ,
\qquad D\odot E\longmapsto D\cdot E.
\]
Here $\odot$ denotes multiplication in the symmetric algebra,
whereas $\cdot$ denotes the intersection product modulo numerical
equivalence. Each $Q\in\Sym^2N^1(X)_\RR$ has a unique expression
\[
Q=\sum_{i,j=1}^r q_{ij}D_i\odot D_j,
\qquad q_{ij}=q_{ji}.
\]
Via the intersection pairing $N^1(X)_\RR\simeq N_1(X)_\RR^*$,
this is the quadratic form
\[
Q(\gamma)=\sum_{i,j=1}^r q_{ij}(D_i\cdot\gamma)(D_j\cdot\gamma)
\quad(\gamma\in N_1(X)_\RR).
\]
We call $Q$ positive definite if the symmetric matrix $(q_{ij})$
is positive definite, equivalently if $Q(\gamma)>0$ for every
nonzero $\gamma\in N_1(X)_\RR$.
For a smooth Fano variety $X$, its pseudoindex $\iota_X$ is the
minimum anticanonical degree among rational curves on $X$.
Our second result characterizes complete flags numerically.

\begin{theorem}\label{thm:quadratic-main}
Let $X$ be a positive-dimensional smooth Fano variety. The following
are equivalent.
\begin{enumerate}[label=\textup{(\arabic*)}]
\item $X$ is a complete flag variety.
\item $\iota_X\geq2$ and $R_X\otimes_\QQ\RR$ contains a
positive-definite quadratic form.
\item $-K_X\in2\Pic(X)$ and $R_X\otimes_\QQ\RR$ contains a
positive-definite quadratic form.
\end{enumerate}
\end{theorem}

The key observation is that a positive-definite quadratic relation
rules out surfaces in the fibers of any contraction of relative Picard
number one (Proposition~\ref{prop:surface-obstruction}). When
$\iota_X\geq2$, standard contraction theorems yield smooth
$\PP^1$-fibrations, and the geometric characterizations in
\cite{OSWW,OSW} apply. Conversely, the quadratic relation on a
complete flag is a classical consequence of Borel's presentation of
its cohomology ring \cite{Borel}.

For Theorem~\ref{thm:odd-degree-main}, odd degree and Poincar\'e duality
give an odd-index embedding of Picard lattices; Wu's formula then shows
that $-K_X$ is divisible by two. Once $X$ is known to be a complete
flag, the Weyl dimension formula identifies the coroot lines from the top
intersection polynomial. Their integral structure modulo two distinguishes types $B_m$ and $C_m$.

\section{Quadratic relations and extremal contractions}\label{sec:mori}

Unless otherwise specified, we follow Hartshorne~\cite{Hartshorne}
and Koll\'ar--Mori~\cite{KM} for terminology and notation in algebraic
geometry; Hatcher~\cite{Hatcher} and Milnor--Stasheff~\cite{MilnorStasheff}
for algebraic topology; and Fulton--Harris~\cite{FultonHarris},
Bourbaki~\cite{Bourbaki}, and Humphreys~\cite{Humphreys} for complex
semisimple Lie algebras, root systems, and Weyl groups.

Throughout, all varieties are defined over $\CC$; homotopy equivalence and topological degree refer to the classical topology. A contraction is a projective surjective morphism to a normal variety with connected fibers. A smooth $\PP^1$-fibration means a smooth morphism whose fibers are isomorphic to $\PP^1$.

\begin{proposition}\label{prop:surface-obstruction}
Let $X$ be smooth projective, and assume that $R_X\otimes_\QQ\RR$ contains a positive-definite form. If $f:X\to Y$ is a contraction with $\rho(X/Y)=1$, then every fiber of $f$ has dimension at most one.
\end{proposition}

\begin{proof}
Write $Q=\sum_{i,j}q_{ij}D_i\odot D_j\in R_X\otimes_\QQ\RR$,
where $(q_{ij})$ is positive definite. Then
\[
 \sum_{i,j}q_{ij}D_i\cdot D_j=0\quad\text{in }N^2(X)_\RR.
\]
Let $M$ be an $f$-ample Cartier divisor. Since $N^1(X/Y)_\RR$
is one-dimensional, there are $a_i\in\RR$ such that
$D_i\equiv_f a_iM$. The vector $a=(a_i)$ is nonzero.

If a fiber contains an integral surface $S$, then
$D_i|_S\equiv a_iM|_S$ numerically. Hence
\[
 0=\sum_{i,j}q_{ij}(D_i\cdot D_j\cdot S)
 =\left(\sum_{i,j}q_{ij}a_ia_j\right)M^2\cdot S>0,
\]
a contradiction. Every projective variety of dimension at least two contains an integral surface.
\end{proof}

Compare \cite[Lemma~2.2]{WeberWisniewski}, which bounds fiber dimensions
by comparing fibrations on homeomorphic varieties. Here the only input
is a definite quadratic relation on $X$.

For a $K_X$-negative extremal ray $R$, let
\[
 \ell(R)=\min\{-K_X\cdot C:[C]\in R,\ C\text{ a rational curve}\}
\]
be its length.

\begin{proposition}\label{prop:P1-fibration}
Let $X$ be smooth projective and suppose that $R_X\otimes_\QQ\RR$ contains a positive-definite form. If a $K_X$-negative extremal ray $R$ has length at least two, then its contraction is a smooth $\PP^1$-fibration.
\end{proposition}

\begin{proof}
Let $f:X\to Y$ be the contraction and $n=\dim X$. By Proposition~\ref{prop:surface-obstruction}, every fiber has dimension at most one. If $f$ were birational, the Ionescu--Wi\'sniewski inequality would give
\[
 \dim E+\dim F\geq n+\ell(R)-1\geq n+1
\]
for an exceptional component $E$ and a positive-dimensional fiber component $F$, contradicting $\dim E\leq n-1$ and $\dim F\leq1$; see
\cite[Theorem~1.1]{WisContractions} and the inequality recalled in
\cite[\S1.B]{HN}.

Thus $f$ is of fiber type, and all fibers have dimension one. A smooth general fiber is $\PP^1$, so $\ell(R)=2$. The equidimensional maximal-length contraction theorem
\cite[Theorem~1.3]{HN} then implies that $f$ is a smooth $\PP^1$-fibration.
\end{proof}

\begin{theorem}\label{thm:independent-rays}
Let $X$ be a positive-dimensional smooth projective variety with $\rho(X)=r$. Assume that $R_X\otimes_\QQ\RR$ contains a positive-definite form and that there are $r$ linearly independent $K_X$-negative extremal rays, each of length at least two. Then $X$ is a complete flag variety.
\end{theorem}

\begin{proof}
By Proposition~\ref{prop:P1-fibration}, the contractions of the
specified rays are smooth $\PP^1$-fibrations. These $r=\rho(X)$
independent contractions satisfy the hypotheses of
\cite[Appendix~A, Theorem~A.1]{OSW}, which gives the assertion.
\end{proof}

For relative flag-recognition results over more general bases, see \cite{GSB}.

We record the relation needed on a flag variety. Let $F=G/B$ and
replace $G$ by its simply connected cover, which does not change $F$.
Let $P_i\supset B$ be the standard parabolic subgroup associated with the
simple root $\alpha_i$. Write $\pi_i:F\to G/P_i$ for the corresponding
projections, $\Gamma_i\simeq\PP^1$ for their fibers, and $K_i$ for
the relative canonical divisors. As in \cite[Sections~2.2--2.3]{OSWW},
$-K_i$ are the simple roots of $\mathfrak g=\operatorname{Lie}(G)$,
and the reflections
\[
 r_i(D)=D+(D\cdot\Gamma_i)K_i
\]
generate its Weyl group $W$. The classes $[\Gamma_i]$ form a basis of
$\Hom(\Pic(F),\ZZ)$, dual to the fundamental line bundles $L_i$ with
$L_i\cdot\Gamma_j=\delta_{ij}$. Put $S=\Sym(\Pic(F)_\QQ)$, and write $S^W_+$ for the subspace of
positive-degree elements of the invariant ring $S^W$. The characteristic
homomorphism $S\to A^*(F)_\QQ$, sending a line bundle to its divisor
class, is surjective and has kernel the ideal $S\cdot S^W_+$. This follows
from Borel's rational cohomology presentation and the identification of the
rational Chow ring with even-degree cohomology via the Bruhat cell
decomposition; see \cite{Borel,Demazure}. Thus
\begin{equation}\label{eq:Chow-coinvariant}
 A^*(F)_\QQ\simeq S/(S\cdot S^W_+).
\end{equation}

The following classical quadratic relation follows from Borel's
presentation \cite{Borel}. A Weyl-invariant quadratic form is also used
in \cite[\S2.5]{WeberWisniewski}.

\begin{lemma}\label{lem:flag-positive}
A complete flag variety has a rational positive-definite quadratic relation in $N^2(F)_\QQ$. Moreover
\[
 -K_F\in2\Pic(F),\qquad \iota_F=2.
\]
\end{lemma}

\begin{proof}
We retain the notation $L_i$ and $\Gamma_i$ introduced above.
Put
\[
 Q=\sum_{w\in W}w\!\left(\sum_i L_i\odot L_i\right)
 \in\Sym^2N^1(F)_\QQ.
\]
This is positive definite and $W$-invariant, so its image in $A^2(F)_\QQ$
vanishes by \eqref{eq:Chow-coinvariant}, hence it vanishes in $N^2(F)_\QQ$.

Since $-K_F\cdot\Gamma_i=2$ for every $i$, the dual-basis property gives
\[
 -K_F=2\sum_iL_i.
\]
The sum is ample, while each $\Gamma_i$ is a rational curve of anticanonical
degree two. Thus $\iota_F=2$.
\end{proof}

\begin{proof}[Proof of Theorem~\ref{thm:quadratic-main}]
Lemma~\ref{lem:flag-positive} gives \textup{(1)}$\Rightarrow$\textup{(2)},\textup{(3)}. If \textup{(3)} holds, every anticanonical degree of a rational curve is a positive even integer, so $\iota_X\geq2$.

Assume \textup{(2)}. The Mori cone of a smooth Fano variety is a full rational polyhedral cone. Choose $\rho(X)$ linearly independent extremal rays. Each has length at least $\iota_X\geq2$, so Theorem~\ref{thm:independent-rays} applies.
\end{proof}

\begin{corollary}\label{cor:fano-free}
Let $X$ be a positive-dimensional smooth projective variety. Then $X$ is a complete flag variety if and only if
\begin{enumerate}[label=\textup{(\roman*)}]
\item $R_X\otimes_\QQ\RR$ contains a positive-definite quadratic form;
\item $-K_X\in2\Pic(X)$;
\item the $K_X$-negative extremal rays span $N_1(X)_\RR$.
\end{enumerate}
\end{corollary}

\begin{proof}
The necessity follows from Lemma~\ref{lem:flag-positive} and the Fano property of a complete flag. Conversely, choose $\rho(X)$ independent $K_X$-negative extremal rays. By \textup{(ii)}, every such ray has even positive length, hence length at least two. Apply Theorem~\ref{thm:independent-rays}.
\end{proof}

\begin{remark}\label{rem:sharpness}
Neither condition \textup{(ii)} nor condition \textup{(iii)} of Corollary~\ref{cor:fano-free} can be omitted. On $\mathbb F_1=\operatorname{Bl}_p\PP^2$, with $H$ the pullback of a line and $E$ the exceptional curve, one has $H^2+E^2=0$ in $N^2(\mathbb F_1)_\QQ$, so
$H\odot H+E\odot E\in R_{\mathbb F_1}$ is positive definite, but $-K_{\mathbb F_1}\cdot E=1$ and $\mathbb F_1$ is not a complete flag.

Likewise, weak Fano positivity does not replace condition \textup{(iii)} of Corollary~\ref{cor:fano-free}. On $\mathbb F_2$, with negative section $S$ and ruling fiber $F$, put $u=S+F$ and $v=F$. Then $u^2=v^2=0$ in $N^2(\mathbb F_2)_\QQ$, and
$u\odot u+v\odot v\in R_{\mathbb F_2}$ is positive definite, while
\[
 -K_{\mathbb F_2}=2(u+v)
\]
is nef and big. However $-K_{\mathbb F_2}\cdot S=0$, and the $K_X$-negative extremal rays do not span $N_1(X)_\RR$.
\end{remark}

\section{Rigidity from maps of odd degree}\label{sec:odd-degree}

For a smooth Fano variety, Kodaira vanishing, the exponential sequence,
and GAGA give
\begin{equation}\label{eq:Pic-H2-Fano}
 \Pic(X)\xrightarrow{\sim}H^2(X,\ZZ).
\end{equation}
Put $\mathbb F_2:=\ZZ/2\ZZ$, the field with two elements.
For any space $Y$, ``$\bmod 2$'' on integral cohomology means the
coefficient-reduction map
$\operatorname{red}_2:H^k(Y,\ZZ)\to H^k(Y,\mathbb F_2)$
induced by $\ZZ\to\mathbb F_2$. For a closed oriented manifold $Y$,
we similarly write $[Y]_2$ for the image of its fundamental class
in homology with $\mathbb F_2$ coefficients.
For the argument with continuous maps, we use Poincar\'e duality
\cite[Theorem~3.30, p.~241]{Hatcher}, the defining identity for Wu
classes and Wu's formula \cite[\S11, Lemma~11.13 and Theorem~11.14]{MilnorStasheff}.
For a complex manifold $Y$, let $w_2(Y)\in H^2(Y,\mathbb F_2)$
denote the \emph{second Stiefel--Whitney class} of its underlying real
tangent bundle. The identity
$w_2(Y)=\operatorname{red}_2(c_1(T_Y))
=\operatorname{red}_2(c_1(-K_Y))$
follows from the corresponding identity for complex vector bundles
\cite[\S14, Problem~14-B]{MilnorStasheff}.

\begin{corollary}\label{cor:truncated}
Let $X$ be a positive-dimensional smooth Fano variety. Suppose that
\[
 \ker\!\left(\Sym^2H^2(X,\RR)\longrightarrow H^4(X,\RR)\right)
\]
contains a positive-definite quadratic form. If any one of the following
holds, then $X$ is a complete flag variety:
\begin{enumerate}[label=\textup{(\roman*)}]
\item $-K_X\in2\Pic(X)$;
\item the second Stiefel--Whitney class $w_2(X)$ vanishes;
\item $\iota_X\geq2$.
\end{enumerate}
\end{corollary}

\begin{proof}
All real degree-two cohomology classes of a smooth Fano variety are
divisor classes by \eqref{eq:Pic-H2-Fano}. Homological equivalence
implies numerical equivalence. Consider the short exact sequence of
coefficient groups
\begin{equation}\label{eq:mod-two-coeff}
0\longrightarrow\ZZ\xrightarrow{\times2}\ZZ
\longrightarrow\mathbb F_2\longrightarrow0.
\end{equation}
The associated long exact cohomology sequence contains
\begin{equation}\label{eq:mod-two-long}
H^2(X,\ZZ)\xrightarrow{\times2}H^2(X,\ZZ)
\xrightarrow{\operatorname{red}_2}H^2(X,\mathbb F_2)
\xrightarrow{\beta}H^3(X,\ZZ),
\end{equation}
where $\beta$ is the connecting (Bockstein) homomorphism.
Exactness gives $\ker(\operatorname{red}_2)=2H^2(X,\ZZ)$.
Since $w_2(X)=\operatorname{red}_2(c_1(-K_X))$, the isomorphism
\eqref{eq:Pic-H2-Fano} shows that conditions \textup{(i)} and
\textup{(ii)} are equivalent. Theorem~\ref{thm:quadratic-main} applies.
\end{proof}

\begin{proof}[Proof of Theorem~\ref{thm:odd-degree-main}]
Set $N=\dim_{\CC}X=\dim_{\CC}F$, $r=\rho(X)=\rho(F)$,
and $d=\deg f\in2\ZZ+1$.  The assertion is immediate for $N=0$,
so assume $N>0$.  We distinguish the topological, intersection-theoretic,
and root-theoretic steps.

\smallskip\noindent
\emph{Step 1: the integral Picard lattices and their odd index.}
For any smooth Fano variety $Y$, the first Chern class gives
$\Pic(Y)\simeq H^2(Y,\ZZ)\simeq\ZZ^{\rho(Y)}$;
see \cite[Chapter~2, Proposition~2.1.2]{IP}.
Since $f$ is only continuous, its action on Picard groups means the
homomorphism
\[
 A:=c_1^{-1}\circ f^*\circ c_1:\Pic(F)\longrightarrow\Pic(X),
\]
not an algebraic pullback of divisors.

We first recall why a map of nonzero degree induces injections on
cohomology with field coefficients.  If $k$ is a field in which
$d\ne0$, and $a\in H^j(F,k)$ satisfies $f^*a=0$, then for any
$b\in H^{2N-j}(F,k)$ the naturality of the cup product and of
evaluation gives
\[
 d\langle a\smile b,[F]_k\rangle
 =\langle f^*(a\smile b),[X]_k\rangle
 =\langle f^*a\smile f^*b,[X]_k\rangle=0.
\]
The Poincar\'e pairing over $k$ is nondegenerate
\cite[Theorem~3.30]{Hatcher}, hence $a=0$.  As $d$ is odd,
we may take both $k=\QQ$ and $k=\mathbb F_2$.
The rational injection in degree two, together with $\rho(X)=\rho(F)$,
shows that $A\otimes\QQ$ is an isomorphism.  Thus $A$ is an
injection between free integral lattices of the same rank,
and its finite index equals $|\det A|$ in integral bases.

It remains to show that this index is odd. Applied to $Y=X,F$,
the coefficient short exact sequence \eqref{eq:mod-two-coeff}
gives, by exactness, natural injections
\[
 j_Y:\Pic(Y)/2\Pic(Y)
 \simeq H^2(Y,\ZZ)/2H^2(Y,\ZZ)
 \hookrightarrow H^2(Y,\mathbb F_2),\qquad Y=X,F.
\]
Let
\[
 \overline A:\Pic(F)/2\Pic(F)\longrightarrow\Pic(X)/2\Pic(X),
 \qquad [L]\longmapsto[A(L)],
\]
be the map induced by $A$.
By the definition of $A$ and the naturality of reduction of
coefficients, the following diagram commutes:
\[
\begin{CD}
\Pic(F)/2\Pic(F) @>{\overline A}>> \Pic(X)/2\Pic(X) \\
@V{j_F}VV @VV{j_X}V \\
H^2(F,\mathbb F_2) @>{f^*}>> H^2(X,\mathbb F_2).
\end{CD}
\]
Consequently,
if $\overline A(v)=0$ for
$v\in\Pic(F)/2\Pic(F)$, then
\[
 f^*j_F(v)=j_X\overline A(v)=0.
\]
The injection of $f^*$ on $H^2(-,\mathbb F_2)$ and that of $j_F$
force $v=0$.  Thus $\overline A$ is injective.  Its source and
target are both $r$-dimensional $\mathbb F_2$-spaces, so it is
an isomorphism. Since $\overline A$ is induced by $A$ modulo $2$,
its injectivity implies that
\[
 {\det A\equiv1\pmod2,\qquad
 [\Pic(X):A\Pic(F)]=|\det A|\text{ is odd}.}
\]

\smallskip\noindent
\emph{Step 2: the second Stiefel--Whitney class.}
For a complex manifold $Y$, its real tangent bundle is the
underlying real bundle of $T_Y$, so
$w_2(Y)=c_1(T_Y)\bmod2=c_1(-K_Y)\bmod2$
\cite[\S14, Problem~14-B]{MilnorStasheff}.
Because $\overline A$ is surjective, one may choose
$L\in\Pic(F)$ and $M\in\Pic(X)$ such that
$-K_X=A(L)+2M$ in additive notation.  Set
$\eta:=c_1(L)\bmod2\in H^2(F,\mathbb F_2)$.
The definition of $A$ and the compatibility of reduction with
$f^*$ give
\[
 f^*\eta=\bigl(c_1(A(L))\bigr)\bmod2
 =c_1(-K_X)\bmod2=w_2(X).
\]
Here is the precise form of Wu's theorem that we need.
For an oriented closed real $2N$-manifold $M$, the Wu classes
$v_i(M)$ are characterized by
$\langle v_i(M)\smile a,[M]_2\rangle
 =\langle\operatorname{Sq}^i(a),[M]_2\rangle$
for $a\in H^{2N-i}(M,\mathbb F_2)$.
The identity $w=\operatorname{Sq}(v)$ gives
$w_2=v_2+\operatorname{Sq}^1(v_1)=v_2+w_1^2$.
Since an oriented manifold has $w_1=0$, we obtain
\begin{equation}\label{eq:wu-for-paper}
 \langle w_2(M)\smile a,[M]_2\rangle
 =\langle\operatorname{Sq}^2(a),[M]_2\rangle
 \quad(a\in H^{2N-2}(M,\mathbb F_2)).
\end{equation}
These are the defining property of the Wu classes and Wu's formula
\cite[\S11, Lemma~11.13 and Theorem~11.14]{MilnorStasheff}.
Unlike tangent bundles, Steenrod squares are natural under
\emph{every} continuous map:
$\operatorname{Sq}^2(f^*b)=f^*(\operatorname{Sq}^2b)$
\cite[\S4.L]{Hatcher}.

For arbitrary $b\in H^{2N-2}(F,\mathbb F_2)$ we now compute
\begin{align*}
 \langle\eta\smile b,[F]_2\rangle
 &=\langle f^*(\eta\smile b),[X]_2\rangle\\
 &=\langle w_2(X)\smile f^*b,[X]_2\rangle\\
 &=\langle\operatorname{Sq}^2(f^*b),[X]_2\rangle\\
 &=\langle f^*(\operatorname{Sq}^2b),[X]_2\rangle\\
 &=\langle\operatorname{Sq}^2b,[F]_2\rangle\\
 &=\langle w_2(F)\smile b,[F]_2\rangle=0.
\end{align*}
The first and fifth equalities follow from
$f_*[X]_2=[F]_2$, because $d\equiv1\pmod2$;
the second uses $f^*\eta=w_2(X)$ and multiplicativity of
$f^*$; the third and sixth are \eqref{eq:wu-for-paper};
and the fourth is naturality of $\operatorname{Sq}^2$.
Finally, $w_2(F)=0$ follows from
$-K_F=2\sum_iL_i$ (Lemma~\ref{lem:flag-positive}).
Since this holds for every $b$, the perfect Poincar\'e pairing
on $F$ forces $\eta=0$, and consequently $w_2(X)=0$.
By exactness of \eqref{eq:mod-two-long}, the kernel of
$\operatorname{red}_2:H^2(X,\ZZ)\to H^2(X,\mathbb F_2)$ is
$2H^2(X,\ZZ)$.
Applying the isomorphism $c_1:\Pic(X)\simeq H^2(X,\ZZ)$ gives
\[
 {-K_X\in2\Pic(X).}
\]

\smallskip\noindent
\emph{Step 3: transfer of the quadratic relation.}
Use the fundamental line bundles $L_1,\ldots,L_r$ as a rational
basis of $N^1(F)_\QQ$. By Lemma~\ref{lem:flag-positive}, there is
a positive-definite quadratic form
$Q_F=\sum_{i,j}q_{ij}L_i\odot L_j$ with $(q_{ij})$ symmetric,
whose corresponding cup product in $H^4(F,\QQ)$ vanishes. Put
$Q_X=\sum_{i,j}q_{ij}A(L_i)\odot A(L_j)$.
The corresponding cup product on $X$ also vanishes:
\[
 \sum_{i,j}q_{ij}c_1(A(L_i))\smile c_1(A(L_j))
 =f^*\!\left(\sum_{i,j}q_{ij}c_1(L_i)\smile c_1(L_j)\right)=0.
\]
Since homological equivalence implies numerical equivalence,
the corresponding relation also holds in $N^2(X)_\QQ$
\cite[Chapter~19]{Fulton}.
Thus $\mu_X(Q_X)=0$, i.e., $Q_X\in R_X\otimes_\QQ\RR$.
Since $A\otimes\RR$ is invertible, $A(L_1),\ldots,A(L_r)$
is a real basis; in that basis the coefficient matrix of $Q_X$
is $(q_{ij})$, so $Q_X$ remains positive definite.
Together with Step~2, Theorem~\ref{thm:quadratic-main}
shows that $X$ is a complete flag variety.

\Needspace{5\baselineskip}
\smallskip\noindent
\emph{Step 4: the coroot factors of the intersection polynomial.}
Write $F=F_\Phi$ and $X=F_\Psi$.
We use the term \emph{root system} in the sense of
\cite[\S9.2]{Humphreys}: the definition already requires that
only $\alpha$ and $-\alpha$ occur on the line of each root
and that all Cartan integers are integral. As in that reference,
$\Theta^\vee$ denotes the dual (or inverse) root system of
$\Theta$; its roots will be called \emph{coroots}.
For either flag $Y=F_\Theta$, write
$P_\Theta:=\Pic(Y)$ for its weight lattice and
$Q_\Theta^\vee:=\Hom(P_\Theta,\ZZ)$ for its coroot lattice.
In \cite[\S13.1]{Humphreys}, the weight lattice is denoted
$\Lambda$; in this paper the symbol $Q_\Theta^\vee$ is used
for the coroot lattice to avoid confusion.
Write $\mathcal C_Y^+$ for the positive coroots. The classes
of the standard $\PP^1$-fibers $[\Gamma_i]$ form an integral
basis of $Q_\Theta^\vee$. The other coroots are their
Weyl-group translates; see \cite[\S\S2.2--2.3]{OSWW}.
Put $H_Y=-K_Y/2$. We derive the factorization from the
Borel--Weil theorem and the Weyl dimension formula; see
\cite[Theorems~1--2]{LurieBWB} and
\cite[Corollary~24.6, p.~402]{FultonHarris}.
Passing to the simply connected cover of the defining group does
not change $Y$.  Its Picard lattice is identified with the weight
lattice, and a dominant integral weight $\lambda$ corresponds to a
globally generated line bundle $L_\lambda$, with
$H^0(Y,L_\lambda)\simeq V_\lambda^*$ by Borel--Weil.
Here $V_\lambda$ is the irreducible representation of highest
weight $\lambda$.  Write
$\varrho=\frac12\sum_{\alpha\in\Theta^+}\alpha
=\sum_i\omega_i$, where the $\omega_i$ are the fundamental weights.
Under the weight--Picard correspondence, $H_Y$ corresponds to
$\varrho$, since $-K_Y=2\sum_i L_i$.

Suppose first that $D=L_\lambda$ for a dominant integral
weight $\lambda$.  Since $mD$ is nef and $-K_Y$ is ample,
$mD-K_Y$ is ample for every $m\ge0$.  Kodaira vanishing gives
$H^j(Y,\mathcal O_Y(mD))=0$ for $j>0$.  Hence Borel--Weil
and the Weyl dimension formula imply
\[
\begin{aligned}
 \chi(Y,\mathcal O_Y(mD))
 &=\dim V_{m\lambda}\\
 &=\prod_{\alpha\in\Theta^+}
 \frac{\alpha^\vee(m\lambda+\varrho)}
      {\alpha^\vee(\varrho)}\\
 &=\prod_{\gamma\in\mathcal C_Y^+}
 \left(1+m\frac{D\cdot\gamma}{H_Y\cdot\gamma}\right).
\end{aligned}
\]
Here a coroot is regarded as a linear functional on the weight space.
The last equality uses the natural weight--coroot pairing, written as
evaluation: $\gamma(\lambda)=D\cdot\gamma$ and
$\gamma(\varrho)=H_Y\cdot\gamma$.
There are exactly $N=\dim Y$ positive roots.  Comparing the
coefficient of $m^N$ with the leading term
$\chi(Y,\mathcal O_Y(mD))=(D^N/N!)m^N+O(m^{N-1})$
from Hirzebruch--Riemann--Roch gives
\begin{equation}\label{eq:Weyl-volume}
 \frac{D^N}{N!}
 =\prod_{\gamma\in\mathcal C_Y^+}
 \frac{D\cdot\gamma}{H_Y\cdot\gamma}.
\end{equation}
Writing $D=\sum_{i=1}^r a_iL_i$ in the basis of fundamental
line bundles, both sides of \eqref{eq:Weyl-volume} are polynomials
in $a_1,\ldots,a_r$.  Their difference vanishes for every
$(a_1,\ldots,a_r)\in\ZZ_{\ge0}^r$, and hence vanishes identically:
this follows by induction on $r$ from the fact that a nonzero
one-variable polynomial has only finitely many roots.  Thus
\eqref{eq:Weyl-volume} holds for every divisor class $D$.
Alternatively, the same factorization follows from
\cite[Theorem~2.24, p.~584]{OSWW} by comparing the highest-degree
term of its Euler-characteristic formula with Riemann--Roch.
Every factor $D\mapsto D\cdot\gamma$ is a nonzero rational
linear form.

\smallskip\noindent
\emph{Step 5: comparison under the continuous map.}
For any $D\in\Pic(F)$, the cycle class map identifies its
algebraic top intersection number with
$\langle c_1(D)^N,[F]\rangle$, and similarly for $X$
\cite[Chapter~19]{Fulton}. The definition of $A$, multiplicativity
of cup products, and $f_*[X]=d[F]$ therefore give
\begin{align*}
 (AD)^N
 &=\langle c_1(AD)^N,[X]\rangle
 =\langle f^*(c_1(D)^N),[X]\rangle\\
 &=\langle c_1(D)^N,f_*[X]\rangle=dD^N.
\end{align*}
This uses only the continuity of $f$, not its algebraicity.
Since the equality holds for every integral divisor class, it is
an identity of homogeneous polynomials over $\QQ$: a polynomial
vanishing on all points of $\ZZ^r$ is identically zero.

The dual homomorphism of integral lattices is
\[
 A^\vee:Q_\Psi^\vee\longrightarrow Q_\Phi^\vee,
 \qquad A^\vee(\gamma)=\gamma\circ A.
\]
In particular, for $D\in\Pic(F_\Phi)$ and
$\gamma\in\mathcal C_\Psi^+$, one has
$(AD)\cdot\gamma=D\cdot A^\vee\gamma$.
Using \eqref{eq:Weyl-volume} for both flag varieties, and
absorbing the nonzero denominators into constants
$c_\Phi,c_\Psi\in\QQ^\times$, we obtain
\begin{equation}\label{eq:coroot-factor-comparison}
 c_\Psi\prod_{\gamma\in\mathcal C_\Psi^+}
      (D\cdot A^\vee\gamma)
 =d\,c_\Phi\prod_{\delta\in\mathcal C_\Phi^+}
      (D\cdot\delta).
\end{equation}

Choose rational coordinates $x_1,\ldots,x_r$ on
$\Pic(F_\Phi)_\QQ$. Then \eqref{eq:coroot-factor-comparison}
is an equality in $\QQ[x_1,\ldots,x_r]$.
For any positive coroot $\gamma$, the associated linear factor is
$\ell_\gamma(D)=D\cdot\gamma$. Two such factors are proportional
if and only if their coroots span the same rational line.
By the root-system axiom just recalled, the only coroots on
the line $\QQ\gamma$ are $\gamma$ and $-\gamma$, exactly one
of which is positive. Thus positive coroots correspond bijectively to
coroot lines.

Every $\ell_\gamma$ is a nonzero linear polynomial, hence
irreducible. The factors on the right of
\eqref{eq:coroot-factor-comparison} are pairwise nonassociate
by the preceding observation; the same holds on the left since
$A^\vee\otimes\QQ$ is an isomorphism (Step~1).
Unique factorization therefore matches the factors on both sides
up to a permutation and nonzero rational multiples. It gives a
bijection of coroot lines, and hence a bijection of positive
coroots $\sigma:\mathcal C_\Psi^+\to\mathcal C_\Phi^+$.
However, $A^\vee$ need not carry a coroot to a coroot:
the corresponding linear factors are determined only up to
nonzero scalars. Accordingly, for each
$\gamma\in\mathcal C_\Psi^+$ there is a number
$b_\gamma\in\QQ^\times$ such that
\[
 D\cdot A^\vee\gamma
 =b_\gamma(D\cdot\sigma(\gamma))
 \quad\text{for every }D\in\Pic(F_\Phi)_\QQ.
\]
As the pairing with divisor classes is nondegenerate, this
means
\begin{equation}\label{eq:coroot-line-correspondence}
 A^\vee\gamma=b_\gamma\sigma(\gamma).
\end{equation}
Thus $A^\vee\otimes\QQ$ identifies the sets of rational
coroot lines of $\Psi$ and $\Phi$, but this does not yet give
an isomorphism of dual root systems. The integral lattice
information will be used in Step~7.

\smallskip\noindent
\emph{Step 6: recovering the Coxeter graphs.}
For $\Theta=\Phi,\Psi$, write $V_\Theta=\Pic(F_\Theta)_\RR$ and let
\[
 H_\gamma^\Theta=\{D\in V_\Theta:D\cdot\gamma=0\},\qquad
 \mathcal H_\Theta=\{H_\gamma^\Theta:
             \gamma\in\mathcal C_\Theta^+\}.
\]
These are the reflecting hyperplanes of the Weyl group. By
\eqref{eq:coroot-line-correspondence},
$A^{-1}(H_\gamma^\Psi)=H_{\sigma(\gamma)}^\Phi$.
Thus $A$ identifies the two reflection arrangements and sends
Weyl chambers, together with their walls, to Weyl chambers and
their walls.

We recover the Coxeter graph of the root system directly from
the reflection arrangement. Fix a base and let $C$ be its
fundamental Weyl chamber. Let $H_i,H_j$ be two distinct walls of $C$,
let $s_i,s_j$ be the corresponding reflections, and set
$E=H_i\cap H_j$ and $m_{ij}=\operatorname{ord}(s_is_j)$.
Choose $x\in\overline C\cap E$ lying on no wall of $C$
other than $H_i,H_j$.
By \cite[\S10.3, Lemma~B]{Humphreys}, the stabilizer of a
point in the closure of the fundamental Weyl chamber is generated by the
simple reflections fixing that point. Hence
\[
 \operatorname{Stab}_W(x)=\langle s_i,s_j\rangle.
\]
Since $s_i,s_j$ fix $E$ pointwise and every element fixing
$E$ pointwise fixes $x$, the \emph{pointwise stabilizer}
$W_E:=\{w\in W:w|_E=\operatorname{id}_E\}$ equals
$\langle s_i,s_j\rangle$.

Put $R=s_is_j$. The relations $R^{m_{ij}}=1$ and
$s_iRs_i=R^{-1}$ show that the elements of $W_E$ are
\[
 R^k\quad\text{and}\quad s_iR^k,
 \qquad 0\le k<m_{ij}.
\]
On the two-dimensional orthogonal complement of $E$,
the first $m_{ij}$ elements are rotations and the other
$m_{ij}$ are reflections. Thus $W_E$ is a dihedral group
with precisely $m_{ij}$ reflections
\cite[\S9.4, Exercises~3--4]{Humphreys}.
A reflection fixes $E$ pointwise if and only if its
reflecting hyperplane contains $E$. Consequently,
\begin{equation}\label{eq:coxeter-incidence}
 m_{ij}=\#\{H\in\mathcal H_\Theta:E\subseteq H\}.
\end{equation}
We now explain how this counting formula determines the Coxeter graph.
Choose a chamber $C_\Phi$ and label its walls
$H_1^\Phi,\ldots,H_r^\Phi$. Since $A$ identifies the arrangements,
$C_\Psi:=A(C_\Phi)$ is a chamber with walls
$H_i^\Psi:=A(H_i^\Phi)$. For $\Theta=\Phi,\Psi$, let
$s_i^\Theta$ be the reflection with hyperplane $H_i^\Theta$, and put
\[
 E_{ij}^\Theta:=H_i^\Theta\cap H_j^\Theta,
 \qquad m_{ij}^\Theta:=\operatorname{ord}(s_i^\Theta s_j^\Theta)
 \quad(i\ne j).
\]
Because $A$ is a linear isomorphism,
\[
 A(E_{ij}^\Phi)=E_{ij}^\Psi,
 \qquad E_{ij}^\Phi\subseteq H
 \ \Longleftrightarrow\ E_{ij}^\Psi\subseteq A(H).
\]
Thus $H\mapsto A(H)$ gives a bijection between the reflecting
hyperplanes containing $E_{ij}^\Phi$ and those containing
$E_{ij}^\Psi$. Applying \eqref{eq:coxeter-incidence} on both sides gives
\[
 m_{ij}^\Phi
 =\#\{H\in\mathcal H_\Phi:E_{ij}^\Phi\subseteq H\}
 =\#\{H'\in\mathcal H_\Psi:E_{ij}^\Psi\subseteq H'\}
 =m_{ij}^\Psi.
\]
In particular, the orders of the products of corresponding wall
reflections agree. This conclusion uses only containment relations
among hyperplanes.

To relate these orders to the graph in
\cite[\S11.2]{Humphreys}, choose simple roots
$\alpha_1,\ldots,\alpha_r$ for a chamber. With respect to a
Weyl-invariant positive-definite inner product $(\ ,\ )$, their
Cartan integers are
\[
 c_{ij}:=\alpha_j^\vee(\alpha_i)
 =\frac{2(\alpha_i,\alpha_j)}{(\alpha_j,\alpha_j)}.
\]
Thus $c_{ij}$ is the number denoted
$\langle\alpha_i,\alpha_j\rangle$ in
\cite[\S\S9.1, 11.1]{Humphreys}; the angle brackets there do not
denote the inner product. In Humphreys's convention, the
\emph{Coxeter graph} has one vertex for each simple root, or
equivalently each chamber wall, and $c_{ij}c_{ji}$ edges between
vertices $i$ and $j$, without arrows. In the two-dimensional space
$E^\perp$ used above, the $m_{ij}$ reflecting lines divide the plane
into $2m_{ij}$ congruent sectors, and the two walls bound one of
them. Its angle is therefore $\pi/m_{ij}$. The corresponding
simple roots are inward normals
to its sides, so their angle is $\pi-\pi/m_{ij}$. Hence
\cite[\S9.4]{Humphreys}
\[
 c_{ij}c_{ji}=4\cos^2(\pi/m_{ij}),
 \qquad
 \begin{array}{c|cccc}
 m_{ij}&2&3&4&6\\ \hline
 c_{ij}c_{ji}&0&1&2&3
 \end{array}
\]
Thus equality of the $m_{ij}$ gives equality of the number of edges
between each pair of corresponding vertices. The Coxeter graphs,
including their connected components, are therefore isomorphic.
For the relation between the intersection lattice of a reflection
arrangement and the poset of parabolic subgroups, see
\cite[Theorem~3.1]{BarceloIhrig}.

Here ``connected components'' means connected components of the
Coxeter graph. They correspond precisely to the irreducible
components in the orthogonal decomposition of the root system.
By the classification \cite[\S11.4]{Humphreys}, the graph of each
irreducible component determines its type except that $B_m$ and
$C_m$ have the same graph for $m\ge3$;
$B_2$ and $C_2$ are isomorphic root systems. Indeed, the graph
records only the products $c_{ij}c_{ji}$: on a double edge it
cannot distinguish $(c_{ij},c_{ji})=(-1,-2)$ from $(-2,-1)$.
Consequently the irreducible factors of $\Phi$ and $\Psi$ have
the same types up to possible exchanges $B_m\leftrightarrow C_m$.
Ruling out these exchanges is the purpose of Step~7.

\smallskip\noindent
\emph{Step 7: excluding $B_m/C_m$ by odd index.}
We now use the oddness of $\det A$. Choose fundamental chambers
$C_\Phi^0,C_\Psi^0$. The arrangement isomorphism $A$ from
Step~6 sends $C_\Phi^0$ onto some chamber of $\Psi$.
Since $W_\Psi$ acts transitively on its chambers, an element
$w\in W_\Psi$ sends that image onto $C_\Psi^0$.
For this step only, replace $A$ by $w\circ A$.
The Weyl group preserves the weight lattice
\cite[\S13.2, opening paragraph]{Humphreys}.
Thus both $w$ and $w^{-1}$ act integrally on $P_\Psi$,
so $\det w=\pm1$.
The replacement still identifies coroot lines and has odd
determinant. We use it only to align chambers; the intersection
polynomial identity of Step~5 is not needed after this replacement.
Since $w$ permutes coroot lines, the line correspondence
\eqref{eq:coroot-line-correspondence} still holds for the
replacement, with $\sigma$ and $b_\gamma$ relabeled if necessary.

Let $\alpha_i,\beta_i$ be the simple roots for the two chosen
chambers. After relabeling the simple roots of $\Psi$ according
to the images of the walls, we have
\[
 H_{\alpha_i^\vee}^\Phi=A^{-1}(H_{\beta_i^\vee}^\Psi),
 \qquad
 \ker(A^\vee\beta_i^\vee)=\ker(\alpha_i^\vee).
\]
Nonzero linear forms with the same kernel are proportional:
both factor through the same one-dimensional quotient
of the weight space by their common kernel. Consequently,
\[
 A^\vee(\beta_i^\vee)=a_i\alpha_i^\vee
 \qquad(1\le i\le r).
\]
Since the simple coroots form integral bases of the two coroot
lattices, $a_i\in\ZZ$. Evaluating at $D\in C_\Phi^0$ gives
$\beta_i^\vee(AD)=a_i\alpha_i^\vee(D)>0$, whence $a_i>0$.
Thus, after matching the chambers and ordering their walls,
the matrix of $A^\vee$ in these simple-coroot bases is
$\operatorname{diag}(a_1,\ldots,a_r)$.

To compare determinants, let $\omega_i,\eta_j$ be the fundamental
weights dual to $\alpha_i^\vee,\beta_j^\vee$, respectively.
If $A\omega_i=\sum_j\mu_{ji}\eta_j$ and $M=(\mu_{ji})$, then
\[
 (A^\vee\beta_j^\vee)(\omega_i)
 =\beta_j^\vee(A\omega_i)=\mu_{ji},
 \qquad
 A^\vee\beta_j^\vee=\sum_i \mu_{ji}\alpha_i^\vee.
\]
Consequently the matrices of $A$ and $A^\vee$ in these dual
bases are $M$ and ${}^{t}\!M$, respectively. In particular,
$\det A^\vee=\det A$, even though the arrows of these two
homomorphisms point in opposite directions.
Changing integral bases affects determinants only by a sign, so
\[
 |\det A|=|\det A^\vee|=\prod_{i=1}^r a_i.
\]
Since this product is odd, every $a_i$ is a positive odd integer.

A coroot is \emph{primitive} in its coroot lattice:
for the dual root system, every coroot is Weyl-conjugate to
a simple coroot \cite[\S10.3, Theorem~(c)]{Humphreys},
and the simple coroots form an integral basis of the coroot
lattice \cite[\S13.1]{Humphreys}.
For $\gamma\in\mathcal C_\Psi^+$ put $\delta=\sigma(\gamma)$.
Since $\delta$ is primitive in $Q_\Phi^\vee$, it extends to
an integral basis. Expressing the integral vector
$A^\vee\gamma=b_\gamma\delta$ in this basis shows that
$b_\gamma\in\ZZ$. Since $\det A$ is odd, $A^\vee$
induces an isomorphism of coroot lattices modulo~$2$.
As a primitive coroot is nonzero modulo~$2$ and the reduction
of $A^\vee$ is injective, $b_\gamma$ must be odd. Therefore
\[
 A^\vee\gamma\equiv\sigma(\gamma)\pmod{2Q_\Phi^\vee}.
\]
For a connected component $I$ of the common Coxeter graph,
let $Q_{\Phi,I}^\vee$ and $Q_{\Psi,I}^\vee$ be the integral
spans of the corresponding simple coroots. Since
$A^\vee\beta_i^\vee=a_i\alpha_i^\vee$ with each $a_i$ odd,
$A^\vee$ restricts to a map
$Q_{\Psi,I}^\vee\to Q_{\Phi,I}^\vee$ which is an
isomorphism modulo~$2$. For any positive coroot $\gamma$
of the $I$-component of $\Psi$, the vector $A^\vee\gamma$
belongs to $Q_{\Phi,I}^\vee$. Thus
$\sigma(\gamma)$ also belongs to this component;
the bijection $\sigma$ preserves the component decomposition.
Since $b_\gamma$ is odd, the reduction of $A^\vee$
sends the set of residues of positive coroots of the
$I$-component of $\Psi$ \emph{bijectively as a set} onto
the corresponding set for $\Phi$. Their numbers of distinct
mod-$2$ classes are therefore equal.

For a root system $R$, define its \emph{root lattice} and the
number of distinct classes of positive roots by
\[
 Q(R):=\ZZ R,\qquad
 N_2(R):=\#\{\gamma+2Q(R):\gamma\in R^+\}.
\]
Here $R^+$ denotes any choice of positive roots; $N_2(R)$ is
independent of this choice because $\gamma\equiv-\gamma$
modulo~$2Q(R)$. Let $e_1,\ldots,e_m$ be the standard basis of
$\RR^m$. 
The standard coordinates for $B_m$ and $C_m$
\cite[\S12.1, p.~64]{Humphreys} yield
\begin{align*}
 Q(B_m)&=\ZZ^m,\qquad
 N_2(B_m)=m+\binom m2=\frac{m(m+1)}2,\\
 Q(C_m)&=\{(z_i)\in\ZZ^m:\sum_i z_i\text{ is even}\},\\
 N_2(C_m)&=1+2\binom m2=m(m-1)+1.
\end{align*}
Indeed, in $B_m$ the roots $e_i+e_j$ and $e_i-e_j$ have
the same reduction, whereas in $C_m$ they have distinct
reductions and all $2e_i$ have the same reduction.
The difference $N_2(C_m)-N_2(B_m)=(m-1)(m-2)/2$ is positive
for $m\ge3$. Applying this to the \emph{dual root systems}
($B_m^\vee=C_m$ and $C_m^\vee=B_m$), we see that no
$B_m/C_m$ exchange is possible.

The root systems of $X$ and $F$ are thus isomorphic, including
their irreducible factors. A central isogeny induces an
isomorphism of complete flag varieties, so the root system
determines the corresponding complete flag variety up to
isomorphism. Consequently $X\cong F$.
\end{proof}

\begin{remark}
Cuntz observes that the crystallographic arrangements of types
$B_m$ and $C_m$ are inequivalent for $m\ge3$
\cite[Example~2.6]{Cuntz}. The equivalence in
\cite[Definition~2.3]{Cuntz}, however, must carry the whole
set of root normals onto the other set. Our map $A^\vee$
initially identifies only coroot \emph{lines}, possibly multiplying
individual coroots by nonunit odd integers. Therefore Cuntz's
observation does not by itself rule out an exchange in the
present setting; the mod-$2$ argument in Step~7 does so.
\end{remark}

\begin{corollary}\label{cor:homotopy}
If a smooth Fano variety $X$ is homotopy equivalent to a complete flag
variety $F$, then $X\cong F$.
\end{corollary}

\begin{proof}
A homotopy equivalence has degree $\pm1$ and induces an isomorphism
on $H^2(-,\QQ)$, hence the Picard numbers agree. Apply
Theorem~\ref{thm:odd-degree-main}.
\end{proof}

\begin{remark}[A shortcut for degree $\boldsymbol{\pm1}$ maps]
\label{rem:degree-one}
The stronger odd-degree theorem requires Steps~6--7, but these
steps can be omitted when $|\deg f|=1$. Indeed, define the
cohomological pushforward $f_!:H^2(X,\ZZ)\to H^2(F,\ZZ)$
by Poincar\'e duality, that is,
$f_!=\operatorname{PD}_F^{-1}\circ f_*\circ\operatorname{PD}_X$.
The projection formula gives
$f_!f^*=(\deg f)\operatorname{id}$.
Thus $f^*$ is a split injection of integral degree-two
cohomology groups. Since $\Pic(F)$ and $\Pic(X)$ are free
abelian groups of the same rank and coincide with these
cohomology groups via $c_1$, the map
$A:\Pic(F)\to\Pic(X)$ is an isomorphism of lattices.
So is its dual $A^\vee$. Every coroot is primitive in the
coroot lattice, and therefore
\eqref{eq:coroot-line-correspondence} forces
$b_\gamma=\pm1$. The bijection of coroot lines from
Step~5 is consequently an isomorphism of dual root systems,
giving $X\cong F$ without reconstructing the Coxeter graph
or separately ruling out $B_m/C_m$ exchanges.
In particular, this shorter argument applies when $f$ is
a homotopy equivalence.
\end{remark}

Related rigidity results for projective spaces and certain homotopy
types appear in \cite{HK,Yau,LW}. Mandal \cite{Mandal} studies
nonzero-degree maps between type-$A$ partial flags, whereas the source
in Theorem~\ref{thm:odd-degree-main} is initially an arbitrary smooth
Fano variety.

\begin{corollary}[Fano deformation rigidity; Weber--Wi\'sniewski]
\label{cor:deformation}
Let $p:\mathcal X\to S$ be a smooth projective morphism with connected
smooth complex base $S$ and nonempty smooth Fano fibers. If one fiber is
isomorphic to a complete flag variety $F$, then every fiber is isomorphic
to $F$.
\end{corollary}

\begin{proof}
Ehresmann's theorem \cite[Chapter~9, Theorem~9.3]{Voisin} gives a differentiable local trivialization of the
underlying proper smooth submersion. Since $S$ is connected, its fibers
are diffeomorphic; in particular, their complex dimensions, second
Betti numbers and hence their Fano Picard numbers coincide. The
homeomorphisms have degree $\pm1$. Apply
Theorem~\ref{thm:odd-degree-main}. This recovers
\cite[Theorem~1.1]{WeberWisniewski}.
\end{proof}

The original proof in \cite{WeberWisniewski} analyzes contractions in
a deformation family, using cohomological restrictions on fiber
dimensions (cf.~\cite[Lemma~2.2]{WeberWisniewski}). Here
Proposition~\ref{prop:surface-obstruction} directly excludes surfaces
in elementary fibers, bypassing that analysis. The shorter argument
also relies on the earlier characterization of flags by smooth
$\PP^1$-fibrations \cite{OSWW,OSW}. For related cohomological
constraints on extremal contractions, see \cite{WisCohomology}.

\begin{corollary}\label{cor:finite}
Let $F=G/B$ be a complete flag variety with no $\PP^1$-factor and
let $X$ be a smooth Fano variety with
$\dim X=\dim F$ and $\rho(X)=\rho(F)$.
Every finite surjective morphism $f:X\to F$ of odd degree is an isomorphism.
\end{corollary}

\begin{proof}
By Theorem~\ref{thm:odd-degree-main}, $X\cong F$. The morphism $f$
thus becomes a finite surjective self-morphism of $F$. The theorem
of Paranjape--Srinivas \cite{PS} on surjective algebraic self-maps
of rational homogeneous spaces shows that such a morphism is an
automorphism when $F$ has no projective-space factor. For a complete
flag variety, such a factor can only be $\PP^1$.
\end{proof}

The exclusion of $\PP^1$-factors is essential: a self-map
$\PP^1\to\PP^1$ of degree three is finite but not an automorphism.

\section*{Acknowledgements}
This work was partially supported by JSPS KAKENHI Grant Number 25K06940.

\section*{AI Disclosure}
The starting point of this research was the author's work on the nef
complexity of klt Fano varieties \cite{WatanabeNefComplexity}. In particular,
the graph-theoretic argument involving mixed-intersection numbers in the
proof of \cite[Proposition~5.2, Step~1]{WatanabeNefComplexity} suggested to the
author that analogous intersection-theoretic data might recover the Dynkin
diagram of a complete flag variety. This question was also motivated by the author's
earlier joint work on Dynkin diagrams and geometric characterizations
of flag varieties \cite{MOSW,OSWW}.

The subsequent research developed through extensive discussions with
OpenAI's ChatGPT, using the \textbf{GPT-5.6 Sol} and \textbf{GPT-6}
models, during September and October 2026. In response to questions and
mathematical directions supplied by the author, ChatGPT suggested
intermediate statements, approaches to proofs, and mathematical
arguments. Through these exchanges, the initial question concerning
mixed intersections evolved into the numerical characterization and
rigidity results established in this paper. ChatGPT also assisted with
examining and revising proofs, searching the literature, checking
references, and preparing the manuscript.

The author critically examined and revised the mathematical arguments
and takes full responsibility for the correctness of the results and
the accuracy of the references.


\begin{thebibliography}{99}
\enlargethispage{4\baselineskip}

\bibitem{BarceloIhrig}
H. Barcelo and E. Ihrig,
\emph{Lattices of parabolic subgroups in connection with hyperplane arrangements},
J. Algebraic Combin. \textbf{9} (1999), no.~1, 5--24.

\bibitem{Borel}
A. Borel,
\emph{Sur la cohomologie des espaces fibr\'es principaux et des espaces
homog\`enes de groupes de Lie compacts},
Ann. of Math. (2) \textbf{57} (1953), 115--207.

\bibitem{Bourbaki}
N. Bourbaki,
\emph{Lie Groups and Lie Algebras: Chapters 4--6},
Springer, 2002.

\bibitem{Cuntz}
M. Cuntz,
\emph{Crystallographic arrangements: Weyl groupoids and simplicial arrangements},
Bull. Lond. Math. Soc. \textbf{43} (2011), no.~4, 734--744.

\bibitem{Demazure}
M. Demazure,
\emph{Invariants sym\'etriques entiers des groupes de Weyl et torsion},
Invent. Math. \textbf{21} (1973), 287--301.

\bibitem{Fulton}
W. Fulton,
\emph{Intersection Theory}, 2nd ed.,
Springer, Berlin, 1998.

\bibitem{FultonHarris}
W. Fulton and J. Harris,
\emph{Representation Theory: A First Course},
Graduate Texts in Mathematics, vol.~129, Springer-Verlag,
New York, 1991.

\bibitem{GSB}
I. Grojnowski and N. Shepherd-Barron,
\emph{Recognising flag varieties and reductive groups},
Proc. London Math. Soc. \textbf{133} (2026), no.~2, e70190.

\bibitem{Hartshorne}
R. Hartshorne,
\emph{Algebraic Geometry}, Graduate Texts in Mathematics, vol.~52,
Springer-Verlag, New York, 1977.

\bibitem{Hatcher}
A. Hatcher, \emph{Algebraic Topology}, Cambridge University Press,
Cambridge, 2002.

\bibitem{HK}
F. Hirzebruch and K. Kodaira,
\emph{On the complex projective spaces},
J. Math. Pures Appl. (9) \textbf{36} (1957), 201--216.

\bibitem{HN}
A. H\"oring and C. Novelli,
\emph{Mori contractions of maximal length},
Publ. Res. Inst. Math. Sci. \textbf{49} (2013), no.~1, 215--228.

\bibitem{Humphreys}
J. E. Humphreys,
\emph{Introduction to Lie Algebras and Representation Theory},
Graduate Texts in Mathematics, vol.~9,
Springer-Verlag, New York, 1972.

\bibitem{IP}
V. A. Iskovskikh and Yu. G. Prokhorov,
\emph{Fano Varieties}, in \emph{Algebraic Geometry V},
Encyclopaedia of Mathematical Sciences, vol.~47,
Springer-Verlag, Berlin, 1999.

\bibitem{KM}
J. Koll\'ar and S. Mori,
\emph{Birational Geometry of Algebraic Varieties},
Cambridge Tracts in Mathematics, vol.~134,
Cambridge University Press, Cambridge, 1998.

\bibitem{LW}
A. S. Libgober and J. W. Wood,
\emph{Uniqueness of the complex structure on K\"ahler manifolds of certain homotopy types},
J. Differential Geom. \textbf{32} (1990), 139--154.

\bibitem{LurieBWB}
J. Lurie,
\emph{A proof of the Borel--Weil--Bott theorem},
expository note;
\url{https://people.math.harvard.edu/~lurie/papers/bwb.pdf}.

\bibitem{Mandal}
M. Mandal,
\emph{Degrees of maps and cohomological rigidity of partial flag manifolds},
arXiv:2608.09183v2 (2026).

\bibitem{MilnorStasheff}
J. W. Milnor and J. D. Stasheff,
\emph{Characteristic Classes},
Annals of Mathematics Studies, vol.~76, Princeton University Press, Princeton, 1974.

\bibitem{MOSW}
R. Mu\~noz, G. Occhetta, L. E. Sol\'a Conde, and K. Watanabe,
\emph{Rational curves, Dynkin diagrams and Fano manifolds with nef tangent bundle},
Math. Ann. \textbf{361} (2015), no.~3--4, 583--609.

\bibitem{OSWW}
G. Occhetta, L. E. Sol\'a Conde, K. Watanabe, and J. A. Wi\'sniewski,
\emph{Fano manifolds whose elementary contractions are smooth $\PP^1$-fibrations: a geometric characterization of flag varieties},
Ann. Sc. Norm. Super. Pisa Cl. Sci. (5) \textbf{17} (2017), 573--607.

\bibitem{OSW}
G. Occhetta, L. E. Sol\'a Conde, and J. A. Wi\'sniewski,
\emph{Flag bundles on Fano manifolds},
J. Math. Pures Appl. (9) \textbf{106} (2016), 651--669.

\bibitem{PS}
K. H. Paranjape and V. Srinivas,
\emph{Self maps of homogeneous spaces},
Invent. Math. \textbf{98} (1989), no.~2, 425--444.

\bibitem{Voisin}
C. Voisin,
\emph{Hodge Theory and Complex Algebraic Geometry I},
Cambridge Studies in Advanced Mathematics, vol.~76,
Cambridge University Press, Cambridge, 2002.

\bibitem{WatanabeNefComplexity}
K. Watanabe,
\emph{Nef complexity of klt Fano varieties: the full spectrum and classification below two},
arXiv:2610.27676 (2026).

\bibitem{WeberWisniewski}
A. Weber and J. A. Wi\'sniewski,
\emph{On rigidity of flag varieties},
Int. Math. Res. Not. IMRN \textbf{2018} (2018), no.~9,
2967--2979.

\bibitem{WisCohomology}
J. A. Wi\'sniewski,
\emph{Cohomological invariants of complex manifolds coming from extremal rays},
Asian J. Math. \textbf{2} (1998), no.~2, 289--301.

\bibitem{WisContractions}
J. A. Wi\'sniewski,
\emph{On contractions of extremal rays of Fano manifolds},
J. Reine Angew. Math. \textbf{417} (1991), 141--157.

\bibitem{Yau}
S.-T. Yau,
\emph{On the Ricci curvature of a compact K\"ahler manifold and the complex Monge--Amp\`ere equation. I},
Comm. Pure Appl. Math. \textbf{31} (1978), 339--411.

\end{thebibliography}
\end{document}